\documentclass{article} 

 \usepackage[usenames,dvipsnames]{pstricks}
 \usepackage{epsfig}
 \usepackage{pst-grad} 
 \usepackage{pst-plot} 
 \usepackage[space]{grffile} 
 \usepackage{etoolbox} 
 \makeatletter 
 \patchcmd\Gread@eps{\@inputcheck#1 }{\@inputcheck"#1"\relax}{}{}
 \makeatother
\usepackage{comment}

\usepackage[usenames,dvipsnames]{pstricks}
\usepackage{epsfig}
\usepackage{pst-grad} 
\usepackage{pst-plot} 
\usepackage[space]{grffile} 
\usepackage{etoolbox} 
\makeatletter 
\patchcmd\Gread@eps{\@inputcheck#1 }{\@inputcheck"#1"\relax}{}{}
\makeatother

\usepackage[overload]{empheq}

 \usepackage[usenames,dvipsnames]{pstricks}
 \usepackage{mathtools}
 \usepackage{epsfig}
 \usepackage{pst-grad} 
 \usepackage{pst-plot} 
 \usepackage[space]{grffile} 
 \usepackage{etoolbox} 
 \makeatletter 
 \patchcmd\Gread@eps{\@inputcheck#1 }{\@inputcheck"#1"\relax}{}{}
 \makeatother

\usepackage{amsmath}
\usepackage{amssymb}
\usepackage{amsthm}
\usepackage[latin1]{inputenc}
     
\usepackage{graphicx}

\usepackage[usenames,dvipsnames]{pstricks}
\usepackage{epsfig}
\usepackage{pst-grad} 
\usepackage{pst-plot} 

\usepackage{ifthen}
\usepackage{color}

\newcommand{\intav}[1]{\mathchoice {\mathop{\vrule width 6pt height 3 pt depth  -2.5pt
\kern -8pt \intop}\nolimits_{\kern -6pt#1}} {\mathop{\vrule width
5pt height 3  pt depth -2.6pt \kern -6pt \intop}\nolimits_{#1}}
{\mathop{\vrule width 5pt height 3 pt depth -2.6pt \kern -6pt
\intop}\nolimits_{#1}} {\mathop{\vrule width 5pt height 3 pt depth
-2.6pt \kern -6pt \intop}\nolimits_{#1}}}

\def\polhk#1{\setbox0=\hbox{#1}{\ooalign{\hidewidth\lower1.5ex\hbox{`}\hidewidth\crcr\unhbox0}}}

\renewcommand{\div}{\operatorname{div}}

\renewcommand{\div}{\operatorname{div}}

\newcommand{\curl}{\operatorname{curl}}

\newtheorem{teo}{Theorem}

\newtheorem{Definition}{Definition}
\newtheorem{Lemma}{Lemma}

\newtheorem{Proposition}{Proposition}
\newtheorem{Remark}{Remark}
\newtheorem{Assumption}{A}

\usepackage{setspace}
\begin{document}

\title{Improved regularity for the $p$-Poisson equation}
\author{Edgard A. Pimentel, Giane C. Rampasso and Makson S. Santos}

\date{\today} 

\maketitle

\begin{abstract}

In this paper we produce new, optimal, regularity results for the solutions to $p$-Poisson equations. We argue through a delicate approximation method, under a smallness regime for the exponent $p$, that imports information from a limiting profile driven by the Laplace operator. Our arguments contain a novelty of technical interest, namely a sequential stability result; it connects the solutions to $p$-Poisson equations with harmonic functions, yielding improved regularity for the former. Our findings relate a smallness regime with improved $\mathcal{C}^{1,1-}$-estimates in the pre\-sence of $L^\infty$-source terms. 
\noindent 
\medskip


\medskip 

\end{abstract}

\vspace{.1in}

\section{Introduction}\label{sec_introduction}

In this paper we examine the regularity of the weak (distributional) solutions to
\begin{equation}\label{eq_main}
	\div\left(\left|Du\right|^{p-2}Du\right)\,=\,f\;\;\;\;\;\mbox{in}\;\;\;\;\;B_1,
\end{equation}
where $p>2$ and $f\in L^\infty(B_1)$.

We prove that, if $f\in L^\infty(B_1)$, the gradient of the solutions is asymptotically Li\-pschitz; the exponent $p$ is supposed to satisfy certain approximation conditions. Our methods combine approximation techniques with preliminary compactness and loca\-lization arguments, very much inspired by ideas introduced in \cite{caffarelli89}.

The $p$-Laplacian operator occupies a prominent role in the study of nonlinear partial differential equations (PDEs). Partly due to its rich nonlinear structure, several deve\-lopments on the topic have been known. We refer the reader to the monograph \cite{lindqvist1} for a fairly general account of the theory.  See also \cite{heinonen1}, \cite{evansnote} and references thereinto.

A distinctive feature of \eqref{eq_main} is the dependence of the diffusivity coefficient $|Du|^{p-2}$ on the parameter $p\geq1$. Indeed, as $p$ varies, the $p$-Laplacian operator finds applications in a number of disciplines, covering a various range of topics; we refer the reader to \cite{evansnote} for a neatly presented discussion on that topic.

The first developments in the regularity theory of the solutions to \eqref{eq_main} appeared in \cite{uralceva}, \cite{uhlenbeck} and \cite{evans82}. In those papers, the authors consider the degenerate setting $p> 2$ and establish H\"older continuity of the gradient for the solutions to \eqref{eq_main}. 

In \cite{DiB}, the author considers a more general variant of the $p$-Laplacian given by
\[
	\div a(x,u,Du)\,=\,b(x,u,Du)\,+\,f(x)\;\;\;\;\;\mbox{in}\;\;\;\;\;B_1,
\]
where $a:\mathbb{R}^{2d+1}\to\mathbb{R}^d$, $b:\mathbb{R}^{2d+1}\to\mathbb{R}$ and $f:B_1\to\mathbb{R}$ satisfy natural conditions. The assumptions on the vector field $a$ are such that it includes \eqref{eq_main} for $p>1$. The main result in the paper is that solutions are of class $\mathcal{C}^{1,\beta}$, locally, for some $\beta\in (0,1)$, unknown. A distinctive feature of \cite{DiB} is that its arguments account for both the degenerate and the singular cases through a unified approach. 

Regularity in the purely singular setting is the subject of \cite{Lewis}. In that paper, the author supposes $1<p<2$ and proves that solutions to \eqref{eq_main} are locally of class $\mathcal{C}^{1,\beta}$ for some $\beta\in(0,1)$, unknown. The methods in this work are purely variational. We also mention the developments reported in \cite{Tolksdorf}.

In \cite{kumimado} the authors establish pointwise estimates for solutions to possibly degenerate equations of the form 
\[
	-\div\,A(Du,x)\,=\,\mu\;\;\;\;\;\mbox{in}\;\;\;\;\;B_1,
\] 
where $A:\mathbb{R}^d\times B_1\to\mathbb{R}^d$ satisfies usual structural conditions and $\mu:B_1\to\mathbb{R}$ is a (Borel) measure with finite mass. The conditions imposed on $A$ accommodate \eqref{eq_main} even in the presence of coefficients. 

By resorting to Wolff-type of potentials associated with $\mu$, the authors unveil a pass-through mechanism transmitting regularity properties from the data to the solutions. As a consequence, they produce both Schauder and Calder\'on-Zygmund estimates, in H\"older and Sobolev spaces, respectively. See further \cite{falamuito1} and \cite{falamuito2}.

A number of progresses in the regularity theory for \eqref{eq_main} have been obtained in the plane $\mathbb{R}^2\sim\mathbb{C}$. This is due to key properties and notions available in complex analysis, such as the complex gradient and quasiregular mappings. 

The first important achievement in this direction appeared in \cite{Iwaniec1}. In that work, the authors prove that $p$-harmonic functions $u:B_1\subset\mathbb{R}^2\to\mathbb{R}$ are such that
\[
	u\,\in\,\mathcal{C}^{k,\alpha}_{loc}(B_1)\,\cap\,W^{2+k,q}_{loc}(B_1),
\]
with
\[
	k\,+\,\alpha\,=\,\frac{1}{6}\left(7\,+\,\frac{1}{p\,-\,1}\,+\,\sqrt{1\,+\,\frac{14}{p\,-\,1}\,+\,\frac{1}{\left(p\,-\,1\right)^2}}\right)
\]
and
\[
	1\,\leq\,q\,<\,\frac{2}{2\,-\,\alpha}.
\]	
Still in the planar setting, recent developments relate to the so-called $\mathcal{C}^{p'}$-regularity conjecture. This important open problem in regularity theory is motivated by the existence of solutions to \eqref{eq_main} in $\mathcal{C}^{1,\frac{1}{p-1}}$. Indeed, the (radial) function 
\[
	u(x):=|x|^{1+\frac{1}{p-1}}
\]
solves 
\[
	\Delta_p u\,\sim\, 1
\] 
in the weak sense. In \cite{lindgrenlindqvist1}, the authors prove that solutions to \eqref{eq_main}, in the case $d=2$, are of class $\mathcal{C}^{1,\frac{1}{p-1}-\varepsilon}$, for every $\varepsilon>0$. In case the source term $f\in L^q(B_1)$, for $q\in(2,\infty)$, the authors obtain sharp regularity of the solutions. 

A proof of the $\mathcal{C}^{p'}$-regularity conjecture in the plane is presented in \cite{teixurb}. In that paper, the authors combine quasiregular estimates \cite{savetheday} with the notion of \emph{small correctors} to perform a tangential analysis \emph{importing regularity from the homogeneous setting}. An ingredient of paramount relevance is a gradient-dependent oscillation control. This estimate endogenously accounts for the regions where the equation potentially degenerates. See also \cite{teitowards}.

The present paper advances the regularity theory for \eqref{eq_main} by \emph{importing information from the Laplace equation}. Our arguments rely on a new approximation strategy, unlocked by a novel \emph{sequential stability result}; see Proposition \ref{prop_stability}. We believe the rationale underlying this proposition can be adapted to a larger class of problems, yielding further information on specific perturbations of the Laplace operator. Our main result reads as follows:

\begin{teo}\label{teo_main}
Let $u\in W^{1,p}(B_1)$ be a weak solution to \eqref{eq_main}. Suppose $f\in L^\infty(B_1)$. Given $\alpha\in(0,1)$, there exists $\varepsilon =  \varepsilon(d,\alpha) >0$ such that, if $p>2$ satisfies a proximity regime of the form
\begin{equation}\label{eq_proxpint}
	\left|p\,-\,2\,\right|\,<\,\varepsilon,
\end{equation}
we have $u\in\mathcal{C}^{1,\alpha}(B_{1/2})$. In addition, there exists $C>0$ such that
\[
	\left\|u\right\|_{\mathcal{C}^{1,\alpha}(B_{1/2})}\,\leq\,C\left(\left\|u\right\|_{L^\infty(B_1)}\,+\,\left\|f\right\|_{L^\infty(B_1)}\right).
\]
\end{teo}

Theorem \ref{teo_main} states that for each $\alpha\in(0,1)$ we can find a smallness regime for $p$, depending on $\alpha$, so that solutions are locally of class $\mathcal{C}^{1,\alpha}$. We emphasize that, as $\alpha\to1$, the approximation regime in \eqref{eq_proxpint} requires $\varepsilon\to 0$.

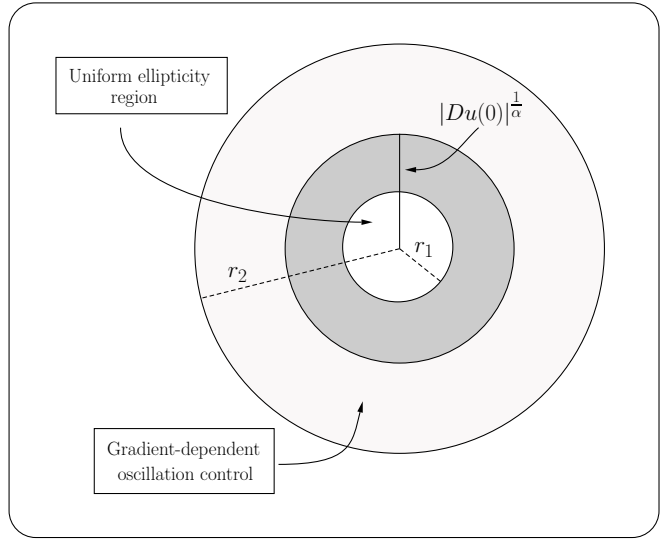
\begin{figure}[h!]
\center
\psscalebox{.350 .350} 
{
\begin{pspicture}(0,-10.181818)(24.727272,10.181818)
\definecolor{colour0}{rgb}{0.98039216,0.972549,0.972549}
\definecolor{colour1}{rgb}{0.8,0.8,0.8}
\pscircle[linecolor=black, linewidth=0.04, fillstyle=solid,fillcolor=colour0, dimen=outer](14.854546,0.81818175){7.8}
\pscircle[linecolor=black, linewidth=0.04, fillstyle=solid,fillcolor=colour1, dimen=outer](14.854546,0.81818175){4.368}
\pscircle[linecolor=black, linewidth=0.04, fillstyle=solid, dimen=outer](14.783636,0.88909084){2.113091}
\psline[linecolor=black, linewidth=0.04](14.854546,0.81818175)(14.854546,5.1861815)
\psline[linecolor=black, linewidth=0.04, linestyle=dashed, dash=0.17638889cm 0.10583334cm](14.854546,0.81818175)(7.3665457,-1.0538182)
\rput[bl](16.33891,5.498182){\Huge{$\left|Du(0)\right|^\frac{1}{\alpha}$}}
\psline[linecolor=black, linewidth=0.04, linestyle=dashed, dash=0.17638889cm 0.10583334cm](16.414545,-0.42981824)(14.854546,0.81818175)
\psbezier[linecolor=black, linewidth=0.04, arrowsize=0.05291667cm 4.0,arrowlength=2.0,arrowinset=0.0]{->}(17.854546,5.418182)(17.154545,4.7781816)(16.654545,3.8181818)(15.054545,3.8181817626953123)
\rput[bl](15.418909,0.5061818){\Huge{$r_1$}}
\rput[bl](8.330909,-0.42981824){\Huge{$r_2$}}
\rput[bl](2.290909,7.018182){\huge{Uniform ellipticity}}
\rput[bl](3.8909092,6.2181816){\huge{region}}
\psbezier[linecolor=black, linewidth=0.04, arrowsize=0.05291667cm 4.0,arrowlength=2.0,arrowinset=0.0]{<-}(13.832142,1.8251342)(11.009739,1.8320866)(4.254545,2.6181817)(4.254545,5.418181762695313)
\rput[bl](3.7454545,-7.163636){\huge{Gradient-dependent}}
\rput[bl](4.1454544,-7.9636364){\huge{oscillation control}}
\psframe[linecolor=black, linewidth=0.04, dimen=outer](10.109091,-6.0)(3.2363636,-8.509091)
\psbezier[linecolor=black, linewidth=0.04, arrowsize=0.05291667cm 4.0,arrowlength=2.0,arrowinset=0.0]{->}(10.254545,-7.3818183)(13.054545,-7.3818183)(13.054545,-6.581818)(13.454545,-4.9818182373046875)
\psframe[linecolor=black, linewidth=0.04, dimen=outer](8.327272,8.109091)(1.7818182,5.927273)
\psframe[linecolor=black, linewidth=0.04, dimen=outer, framearc=0.1](24.727272,10.181818)(0.0,-10.181818)
\end{pspicture}
}
\caption{The value of the gradient at the origin sets the threshold at which the regularity regime switches. Within balls of radius $r_1<\left|Du(0)\right|$, we are in the setting of uniformly elliptic diffusions. Meanwhile, in the case $\left|Du(0)\right|<r_2$ we resort to a gradient-dependent oscillation control based on the notion of $\mathcal{C}^1$-small correctors.}
\end{figure}


The proof of Theorem \ref{teo_main} depends on a delicate iterative method. Our goal is to control the oscillation of the solutions within balls of arbitrary radii $0<r\ll 1$. Given such a radius, we examine two (exclusive) regimes. In fact, either $r>\left|Du(0)\right|^{1/\alpha}$ or otherwise. In the former case, we turn to the notion of \emph{$\mathcal{C}^1$-small correctors}, first introduced in \cite{teixurb}. These structures enable us to produce a finer oscillation control, intrinsically gradient-dependent. 

For the latter, we frame the problem in the theory of uniformly elliptic equations. Here, we also obtain an estimate on the oscillation of the solutions. Once we have exa\-mined both regimes, we verify they are in agreement. Then, we obtain the desired result in the entire domain (locally). For a representation of this gradient-based threshold, and the regime-switching procedure, we refer to the Figure 1.




The remainder of this paper is structured as follows: Section \ref{sec_pmma} gathers a few preliminaries used in the paper and details the main assumptions under which we work. In Section \ref{sec_stabilityapprox}, we examine (sequential) stability properties of the solutions to \eqref{eq_main}. The proof of Theorem \ref{teo_main} is the object of Section \ref{sec_proofteo2}. 

\section{A few preliminaries}\label{sec_pmma}

In what follows we collect basic facts and detail our main assumptions. We start with the definition of weak solution to \eqref{eq_main}, in the distributional sense.

\begin{Definition}[Weak solution]
We say that $u\in W^{1,p}(B_1)$ is a weak solution to \eqref{eq_main} if it verifies
\[
	\int_{B_1}\left|Du\right|^{p-2}Du\cdot D\varphi\,dx\,+\,\int_{B_1}f\varphi \,dx\,=\,0
\]
for every $\varphi\in\mathcal{C}^\infty_0(B_1)$.
\end{Definition}

Our arguments rely on the compactness of the solutions to \eqref{eq_main}. For the sake of completeness, we state it here in the form of a lemma.

\begin{Lemma}[Preliminary compactness]\label{lem_reg}
Let $u\in W^{1,p}(B_1)$ be a weak solution to \eqref{eq_main}. Suppose $f\in L^q(B_1)$ for $q\in(2d,\infty]$. Then, $u\in\mathcal{C}^{1,\beta}_{loc}(B_1)$, for some $\beta\in(0,1)$, and there exists a constant $C>0$ such that
\[
	\left\|u\right\|_{\mathcal{C}^{1,\beta}(B_{1/2})}\,\leq\,C\left(\left\|u\right\|_{L^\infty(B_1)}\,+\,\left\|f\right\|_{L^q(B_1)}\right),
\]
where $C=C(d,p,q)$.
\end{Lemma}
For the proof of Lemma \ref{lem_reg}, we refer the reader to \cite{DiB}; see also \cite{lindqvist1} and the references therein. Next, we detail the assumptions on the data of the problem. We start with a condition on the exponent $p$, confining the analysis to the degenerate setting.

\begin{Assumption}[Degenerate setting]\label{assump_p}
We suppose $p\,>\,2$.
\end{Assumption}

Also, we impose integrability conditions on the source term $f$.	

\begin{Assumption}[Source term]\label{assump_f}
We work under the assumption $f\in L^\infty(B_1)$. 
\end{Assumption}





In the next section we produce a result on the stability of the solutions. In fact, it unlocks a new path connecting \eqref{eq_main} and the Laplace equation.

\section{Sequential stability of weak solutions}\label{sec_stabilityapprox}

When studying the regularity of the solutions to \eqref{eq_main} through geometric techniques, we rely on stability properties. 
We move forward with a proposition on the sequential stability of weak solutions to \eqref{eq_main}.

\begin{Proposition}[Sequential stability of weak solutions]\label{prop_stability}
Let $(p_n)_{n\in\mathbb{N}}\subset\mathbb{R}$ and $(f_n)_{n\in\mathbb{N}}\subset L^\infty(B_1)$ be sequences such that 
\begin{equation*}
	|p_n\,-\,2|\,+\,\left\|f_n\right\|_{L^\infty(B_1)}\,\leq\,\frac{1}{n}.
\end{equation*}
Let $(u_n)_{n\in\mathbb{N}}\subset W^{1,p}(B_1)$ be such that
\begin{equation}\label{eq_stablebesgue}
	\div(|Du_n|^{p_n-2}Du_n)\,=\,f_n \;\;\;\;\; \mbox{in} \;\;\;\;\;B_1.
\end{equation}
Suppose further there exists $u_{\infty}\in \mathcal{C}^1(B_1)$ such that $u_n$ converges to $u_{\infty}$ in $\mathcal{C}^{1}(B_1)$. Then $u_{\infty}$ is a weak solution to 
\[
	\Delta u_{\infty}\,=\,0 \;\;\;\;\; \mbox{in} \;\;\;\;\; B_{9/10}.
\]
\end{Proposition}
\begin{proof}
For every $\varphi\in\mathcal{C}^{\infty}_{0}(B_{1}),$ we have
\begin{align*}
    \left|\int_{B_{9/10}}Du_{\infty}\cdot D\varphi\, dx\right|\,&\leq\,\int_{B_{1}}\left(|D\varphi|\left|Du_{\infty}\,-\,|Du_n|^{p_n-2}Du_n\right|\right)\,dx\,+\,\int_{B_{1}}|f_n||\varphi|\,dx\\
    	&\leq\,C\int_{B_1}\left| Du_{\infty}\,-\,|Du_n|^{p_n-2}Du_n\right|dx\,+\,C\left\|f_n\right\|_{L^{\infty}(B_1)}.
\end{align*}
Lemma \ref{lem_reg} yields
\[
	u_n\,\in\,\mathcal{C}^{1,\beta}_{loc}(B_1)\;\;\;\;\;\mbox{and}\;\;\;\;\;\left\|u_n\right\|_{\mathcal{C}^{1,\beta}(B_{9/10})}\,\leq\,C,
\]
for some $\beta\in(0,1)$ and some $C>0$. We note that neither $\beta$ nor $C$ depend on $n\in\mathbb{N}$; see \cite{DiB,Tolksdorf}. Hence,
\[
	\left|Du_{\infty}\,-\,|Du_n|^{p_n-2}Du_n\right|\,\leq\, \left\|Du_{\infty}\right\|_{L^{\infty}(B_1)}\,+\,\left|Du_n\right|^{p_n-1}\,\leq\,C.
\]
Therefore, since $|Du_n|^{p_n-2}Du_n\rightarrow Du_{\infty}$ a.e.-$x\in B_1$ as $n\rightarrow\infty,$ the Lebesgue Dominated Convergence Theorem implies
\[
	\left|\int_{B_{1}}Du_{\infty}\,\cdot\, D\varphi\,dx\right|\,=\,0,
\]
and the result follows.
\end{proof}

\begin{Remark}[The case $(f_n)_{n\in\mathbb{N}}\subset L^q(B_1)$, $2d<q<\infty$]\label{rem_stablq}\normalfont
Proposition \ref{prop_stability} holds true also in case $(f_n)_{n\in\mathbb{N}}\subset L^q(B_1)$. In fact, it suffices to observe that 
\[
	\int_{B_1}\left|f_n(x)\varphi(x)\right|dx\,\longrightarrow\,0\;\;\;\;\;\mbox{as}\;\;\;\;\;n\to\infty
\]
if $\left\|f_n\right\|_{L^q(B_1)}\,\to\,0$ as $n\to\infty$.
\end{Remark}

\section{Proof of Theorem \ref{teo_main}}\label{sec_proofteo2}

In this section we establish Theorem \ref{teo_main}. In fact, we produce an oscillation control for the solutions to \eqref{eq_main} in balls of radii $0<r\ll 1$.

Our argument explores the interplay between the radius $r$ and the norm of the gradient at the origin. Therefore, the quantity $\left|Du(0)\right|$ sets the threshold for the analysis. For every $\alpha\in(0,1)$, we distinguish the cases $\left|Du(0)\right|<r^\alpha$ and $\left|Du(0)\right|\geq r^\alpha$. Although our focus is at the origin, a change of variables localizes the argument at any point $x_0\in B_1$. We continue with the analysis of the points where the gradient is small. 

\subsection{Oscillation estimates along the critical set}\label{subsec_osccontcrit}

Along the regions where the gradient is small, the oscillation control depends on the notion of \emph{small correctors}. We start by examining the existence of those structures in the context of solutions to \eqref{eq_main}.

\begin{Proposition}[Existence of small correctors]\label{prop_smallcorr}
Let $u\in W^{1,p}(B_1)$ be a bounded weak solution to \eqref{eq_main}. Suppose A\ref{assump_p} and A\ref{assump_f} are in effect. Given $\delta>0$, there exists $\varepsilon_0>0$ such that if
\[
	\left|p\,-\,2\right|\,+\,\left\|f\right\|_{L^\infty(B_1)}\,<\,\varepsilon_0,
\]
then, we can find $\xi\in\mathcal{C}^1(B_{9/10})$ satisfying
\[
	\left|\xi(x)\right|\,<\,\delta\;\;\;\;\;\;\;\;\mbox{and}\;\;\;\;\;\;\;\;\left|D\xi(x)\right|\,<\,\delta
\]
in $B_{9/10}$ and 
\[
	\Delta\left(u\,+\,\xi\right)\,=\,0\;\;\;\;\;\mbox{in}\;\;\;\;\;B_{9/10}.
\]
\end{Proposition}
\begin{proof}
We prove the proposition by contradiction; suppose its statement is false. Then, we can find $\delta_0>0$, a sequence $(p_n)_{n\in\mathbb{N}}$ and sequences of functions $(u_n)_{n\in\mathbb{N}}$ and $(f_n)_{n\in\mathbb{N}}$ such that
\[
	\div\left(\left|Du_n\right|^{p_n-2}Du_n\right)\,=\,f_n\;\;\;\;\;\mbox{in}\;\;\;\;\;B_{9/10},
\]
\[
	\left|p_n\,-\,2\right|\,+\,\left\|f_n\right\|_{L^\infty(B_1)}\,<\,\frac{1}{n},
\]
but for every $\xi\in\mathcal{C}^1(B_{9/10})$ with
\[
	\Delta\left(u_n\,+\,\xi\right)\,=\,0\;\;\;\;\;\mbox{in}\;\;\;\;\;B_{9/10},
\]
we have either 
\[
	\left|\xi(x)\right|\,\geq\,\delta_0\;\;\;\;\;\;\;\;\;\;\mbox{or}\;\;\;\;\;\;\;\;\;\;\left|D\xi(x)\right|\,\geq\,\delta_0,
\]
for some $x\in B_{9/10}$.

Because of Lemma \ref{lem_reg}, we know that $(u_n)_{n\in\mathbb{N}}$ is equibounded in $\mathcal{C}^{1,\beta}(B_{9/10})$, for some $\beta\in(0,1)$. Therefore, there is $u_\infty\in\mathcal{C}^{1,\gamma}(B_{9/10})$, for every $0<\gamma<\beta$, so that $u_n\to u_\infty$ in $\mathcal{C}^{1,\gamma}(B_{9/10})$, through a subsequence, if necessary. By the stability of weak solutions, Proposition \ref{prop_stability}, we infer that
\[
	\Delta u_\infty\,=\,0\;\;\;\;\;\mbox{in}\;\;\;\;B_{9/10}.
\]
Set $\xi_n:=u_\infty-u_n$. It follows from the above that
\[
	\Delta\left(u_n\,+\,\xi_n\right)\,=\,0\;\;\;\;\;\mbox{in}\;\;\;\;\;B_{9/10}
\]
and
\[
	\left|\xi_n(x)\right|\,+\,\left|D\xi_n(x)\right|\,<\,\delta_0,
\]
for every $x\in B_{9/10}$ and $n\gg 1$. This leads to a contradiction and completes the proof.
\end{proof}

\begin{Remark}\normalfont
The fact that $(u_n)_{n\in\mathbb{N}}$ is equibounded in $\mathcal{C}^{1,\beta}_{loc}(B_1)$, for some $\beta\in(0,1)$, follows from the observation that, for $n\gg 1$, we have $p_n\in(2,3)$ and $\left\|f_n\right\|_{L^\infty(B_1)}\ll 1$. That is, the sequences $(p_n)_{n\in\mathbb{N}}$ and $(f_n)_{n\in\mathbb{N}}$ are uniformly bounded. For technical details underlying the argument, we refer the reader to \cite{DiB}; see also \cite[Chapter 2]{Lady}.
\end{Remark}

Once small correctors are available for the solutions to $\eqref{eq_main}$, we produce an oscillation control. This is intrinsically gradient-dependent.

\begin{Proposition}[Oscillation control]\label{prop_osc1}
Let $u\in W^{1,p}(B_1)$ be a bounded weak solution to \eqref{eq_main}. Suppose A\ref{assump_p} and A\ref{assump_f} are in force. Take $\alpha\in(0,1)$, arbitrarily. Then, there exist $\varepsilon_0>0$ and $0<\rho\ll1$ such that, if
\begin{equation}\label{eq_smallregcorr61}
	\left|p\,-\,2\right|\,+\,\left\|f\right\|_{L^\infty(B_1)}\,<\,\varepsilon_0,
\end{equation}
we have
\[
	\sup_{B_\rho}\,\left|u(x)\,-\,u(0)\right|\,\leq\,\rho^{1+\alpha}\,+\,\left|Du(0)\right|\rho.
\]
\end{Proposition}
\begin{proof}
Take $\delta>0$ to be determined further. Suppose $\varepsilon_0>0$ in \eqref{eq_smallregcorr61} is chosen so that there exists a small corrector $\xi\in\mathcal{C}^1(B_1)$ satisfying
\[
	\left|\xi(x)\right|\,<\,\frac{\delta}{3}\;\;\;\;\;\;\;\;\;\;\mbox{and}\;\;\;\;\;\;\;\;\;\;\left|D\xi(x)\right|\,<\,\frac{\delta}{3}
\]
in $B_{9/10}$, with $u\,+\,\xi$ harmonic in $B_{9/10}$. Hence, for $x\in B_\rho$, we have
\begin{align}\label{eq_ctp1}
	\nonumber\left|u(x)-\left[u(0)+Du(0)\cdot x\right]\right|&\leq\left|(u+\xi)(x)-\left[(u+\xi)(0)+D(u+\xi)(0)\cdot x\right]\right|\\ \nonumber
	&\quad+\,\left|\xi(x)\right|\,+\,\left|\xi(0)\right|\,+\,\left|D\xi(0)\,\cdot\, x\right|
	\\&\leq\, C\rho^2\,+\,\delta.
\end{align}
Notice that 
\begin{align*}
	\left\|u\,+\,\xi\right\|_{\mathcal{C}^2(B_{9/10})}\,&\leq\,C\left(\left\|u\right\|_{L^\infty(B_1)}\,+\,\left\|\xi\right\|_{L^\infty(B_{9/10})}\right)\,\leq\,C,
\end{align*}
since we always choose $\delta<1$ and $u$ is a bounded solution. Therefore, the constant $C>0$ in \eqref{eq_ctp1} depends solely on the dimension. 
We then make the universal choices
\[
	\rho\,:=\,\left(\frac{1}{2C}\right)^\frac{1}{1\,-\,\alpha}\;\;\;\;\;\;\;\;\;\;\mbox{and}\;\;\;\;\;\;\;\;\;\;\delta\,:=\,\frac{\rho^{1+\alpha}}{2}
\]
and conclude the proof.
\end{proof}

Finally, we combine Proposition \ref{prop_osc1} with an auxiliary function to unveil a finer oscillation control for $u$ in terms of $\left|Du(0)\right|$.

\begin{Proposition}\label{prop_erre1}
Let $u\in W^{1,p}(B_1)$ be a bounded weak solution to \eqref{eq_main}. Suppose A\ref{assump_p} and A\ref{assump_f} are in effect. Suppose $f\in L^\infty(B_1)$. Given $\alpha\in(0,1)$ there exists $\varepsilon =  \varepsilon(d,\alpha) >0$ to be determined further such that if $p>2$ satisfies a proximity regime of the form
\[
	\left|p\,-\,2\,\right|\,+\,\left\|f\right\|_{L^\infty(B_1)}\,<\,\varepsilon,
\]
then there exists $C>0$ such that
\[
	\sup_{B_r}\,\left|u(x)\,-\,u(0)\right|\,\leq\,Cr^{1+\alpha}\left(1\,+\,\left|Du(0)\right|r^{-\alpha}\right),
\]
for every $0<r\ll 1$.
\end{Proposition}

\begin{proof}
The proof is presented in two steps. We begin by producing an oscillation control at discrete scales of the form $\rho^n$, with $n\in\mathbb{N}$.

\bigskip

\noindent{\bf Step 1}

\bigskip
In what follows we verify that, for every $n\in\mathbb{N}$,
\begin{equation}\label{eq_claimsolvesitall}
	\sup_{B_{\rho^n}}\,\left|u(x)\,-\,u(0)\right|\,\leq\,\left(\rho^{n(1+\alpha)}\,+\,\left|Du(0)\right|\rho^n\right).
\end{equation}

We prove \eqref{eq_claimsolvesitall} by induction in $n\in\mathbb{N}$. The case $n=1$ follows from Proposition \ref{prop_osc1}. Suppose the case $n=k$ has been already verified. We establish the case $n=k+1$. Define $v_k:B_1\to\mathbb{R}$ as follows:
\[
	v_k(x)\,:=\,\frac{u(\rho^kx)\,-\,u(0)}{\rho^{k(1+\alpha)}\,+\,\left|Du(0)\right| \rho^k}.
\]
The induction hypothesis ensures that $v_k$ is bounded in $L^\infty(B_1)$. In addition, 
\[
	Dv_k(x)\,=\,\frac{\rho^kDu(\rho^kx)}{\rho^{k(1+\alpha)}\,+\,\left|Du(0)\right|\rho^k}\;\;\;\;\;\;\mbox{and}\;\;\;\;\;\;\;Dv_k(0)\,=\,\frac{\rho^kDu(0)}{\rho^{k(1+\alpha)}\,+\,\left|Du(0)\right|\rho^k}.
\]
Moreover, $v_k$ is a weak solution to
\[
	\div\left(\left|Dv_k(x)\right|^{p-2}Dv_k(x)\right)\,=\,f_k(x)\;\;\;\;\;\mbox{in}\;\;\;\;\;B_1,
\]
where
\begin{equation}\label{eq_destravalp}
	f_k(x)\,:=\,\frac{\rho^{kp}f(\rho^kx)}{\left(\rho^{k(1+\alpha)}\,+\,\left|Du(0)\right|\rho^k\right)^{p-1}}.
\end{equation}
Since $f\in L^{\infty}(B_1)$, we have that $f_k \in L^{\infty}(B_1)$ provided
\[
	|p\,-\,2|\, \leq\, \frac{1}{\alpha}\,-\,1. 
\]
Set
\[
	\varepsilon^* \,:=\, \min\left\{\varepsilon_0,\frac{1}{\alpha}-1\right\}
\]
and impose
\[
	|p-2|\, \leq\, \frac{1}{2}\varepsilon^*;
\]
then, we can apply Proposition \ref{prop_osc1} to $v_k$. 

Hence, we have
\begin{equation}\label{eq_vktou}
	\sup_{B_{\rho}}\,\left|v_k(x)\,-\,v_k(0)\right|\,\leq\,\rho^{1+\alpha}\,+\,\left|Dv_k(0)\right|\rho.
\end{equation}
The definition of $v_k$ and \eqref{eq_vktou} imply 
\begin{equation}\label{eq_discretev}
	\sup_{B_{\rho^{k+1}}}\,\left|u(x)\,-\,u(0)\right|\,\leq\,\rho^{(k+1)(1+\alpha)}\,+\,\left|Du(0)\right|\rho^{k+1+\alpha}\,+\,\left|Du(0)\right|\rho^{k+1}.
\end{equation}
Because $\rho^\alpha<1$, we have established \eqref{eq_claimsolvesitall}.

\bigskip

\noindent {\bf Step 2}

\bigskip

Next we put forward a discrete-to-continuous argument, built upon \eqref{eq_discretev}. Let $0<r\ll 1$ and take $k\in\mathbb{N}$ such that $\rho^{k+1}\leq r\leq \rho^k$. We have
\begin{align*}
	\sup_{B_r}\,\frac{\left|u(x)\,-\,u(0)\right|}{r^{1+\alpha}}\,&\leq\,\sup_{B_{\rho^k}}\,\frac{\left|u(x)\,-\,u(0)\right|}{\rho^{(k+1)(1+\alpha)}}\\
		&\leq \, \frac{C\left(\rho^{k(1+\alpha)}\,+\,\left|Du(0)\right|\rho^{k}\right)}{\rho^{(k+1)(1+\alpha)}}\\
		&\leq \, \frac{C}{\rho^{1+\alpha}}\left(1\,+\,\frac{|Du(0)|}{\rho^{k\alpha}}\right)\\
		& \leq \,C_\alpha\left(1\,+\,\left|Du(0)\right|r^{-\alpha}\right),
\end{align*}
where 
\[
	C_\alpha\,:=\,\frac{C}{\rho^{1+\alpha}}.
\]
We notice the exponent $\alpha$ is fixed; in addition, the constant $\rho$ is chosen universally. Therefore, $C_\alpha$ depends only on universal quantities and the fixed parameter $\alpha$, which completes the proof.
\end{proof}

In the sequel, we consider balls of radius $r>\left|Du(0)\right|^{1/\alpha}$. We combine intrinsic gradient-dependent estimates with this condition.

\begin{Proposition}\label{prop_gradsmall}
Let $u\in W^{1,p}(B_1)$ be a bounded weak solution to \eqref{eq_main}. Suppose A\ref{assump_p} and A\ref{assump_f} are in effect. Suppose further that
\[
	\left|p\,-\,2\right|\,+\,\left\|f\right\|_{L^\infty(B_1)}\,<\,\varepsilon_0.
\]
Let $\alpha\in(0,1)$. For $\left|Du(0)\right|<r^\alpha\ll 1$, there exists $C>0$ such that
\[
	\sup_{B_r}\,\left|u(x)\,-\,u(0)\right|\,\leq\,Cr^{1+\alpha}.
\]
\end{Proposition}
\begin{proof}
Under the condition $\left|Du(0)\right|<r^\alpha$, Proposition \ref{prop_erre1} yields
\begin{align*}
	\sup_{B_r}\,\left|u(x)\,-\,u(0)\right|\,\leq\,Cr^{1+\alpha}\left(1\,+\,\left|Du(0)\right|r^{-\alpha}\right)\,\leq\,Cr^{1+\alpha},
\end{align*}
and the proof is finished.
\end{proof}

We observe that Proposition \ref{prop_gradsmall} produces an oscillation control in the regions where the gradient of solutions \emph{is small}. To prove Theorem \ref{teo_main} we must also take into account the case of the regions where the gradient is large. This is the topic of the next section.

\subsection{Oscillation control in the non-critical set}\label{subsec_gradlarge}

Here we suppose that $\left|Du(0)\right|\geq r^{\alpha}$. Set $\lambda:=\left|Du(0)\right|^{1/\alpha}$ and define $v:B_1\to\mathbb{R}$ as
\[
	v(x)\,:=\,\frac{u(\lambda x)\,-\,u(0)}{\lambda^{1\,+\,\alpha}}.
\]
We notice that $v(0)=0$, 
\begin{equation}\label{eq_ncritical0}
	\left|Dv(0)\right|=1
\end{equation}
 and
\begin{equation}\label{eq_ncritical1}
	\div\left(\left|Dv(x)\right|^{p-2}Dv(x)\right)\,=\,\frac{\lambda^p\,f(\lambda x)}{\lambda^{(1+\alpha)(p-1)}}\;\;\;\;\;\mbox{in}\;\;\;\;\;B_1.
\end{equation}

The regularity theory available for \eqref{eq_ncritical1}, together with \eqref{eq_ncritical0}, implies that $\left|Dv(x)\right|>1/2$ in $B_\mu$, for some $\mu>0$. Therefore, $v$ solves a uniformly elliptic equation with a bounded source term in $B_\mu$. From standard results in elliptic regularity theory, we infer that $v\in\mathcal{C}^{1,\alpha}(B_{9\mu/10})$, for every $\alpha\in(0,1)$. Hence, for every $\alpha\in(0,1)$, there exists $C>0$ such that
\[
	\sup_{B_r}\,\left|v(x)\,-\,\left[v(0)\,+\,Dv(0)\cdot x\right]\right|\,\leq\,Cr^{1+\alpha},
\]
for every $0<r\ll9\mu/10$. The former inequality reads
\[
	\sup_{B_r}\,\left|\frac{u(\lambda x)\,-\,u(0)}{\lambda^{1+\alpha}}\,-\,\frac{\lambda Du(0)\cdot x}{\lambda^{1+\alpha}}\right|\,\leq\,Cr^{1+\alpha}.
\]
In turn, it leads to 
\[
	\sup_{B_r}\,\left|u(x)\,-\,\left[u(0)\,+\,Du(0)\cdot x\right]\right|\,\leq\,Cr^{1+\alpha},
\]
for $0<r<9\lambda\mu/10$. As concerns the case $9\lambda\mu/10\leq r<\lambda$, we notice that
\begin{align}\label{eq_ing2}
	\nonumber\sup_{B_r}\,\left|u(x)\,-\,\left[u(0)\,+\,Du(0)\cdot x\right]\right|\,&\leq\,\sup_{B_\lambda}\,\left|u(x)\,-\,\left[u(0)\,+\,Du(0)\cdot x\right]\right|\\\nonumber
		&\leq\,\sup_{B_\lambda}\,\left|u(x)\,-\,u(0)\right|\,+\,\left|Du(0)\right|\lambda\\\nonumber
		&\leq\, \left(C\,+\,1\right)\lambda^{1+\alpha}\\\nonumber
		&\leq\, \left(C\,+\,1\right)\left(\frac{10r}{9\mu}\right)^{1+\alpha}\\
		&\leq\, Cr^{1+\alpha}
\end{align}
where the third inequality follows from a Taylor expansion of $u$ along with additional, elementary, facts. 
Now, we present the proof of Theorem \ref{teo_main}. 

\begin{proof}[Proof of Theorem \ref{teo_main}]
The result follows from Proposition \ref{prop_gradsmall} combined with \eqref{eq_ing2}.
\end{proof}

\bibliography{biblio}
\bibliographystyle{plain}

\bigskip

\paragraph{Acknowledgements:} Part of this work was prepared during the visit of the first and the third authors to the International Centre for Theoretical Physics, in Trieste. We are grateful to ICTP for the hospitality. We thank Prof. Eduardo Teixeira and Prof. Jos\'e Miguel Urbano for their insightful, sharp, comments on the material in this paper. We also thank two anonymous referees, whose suggestions improved substantially the present work.  EP is partially supported by CNPq-Brazil (Grants \#433623/2018-7 and \#307500/2017-9), FAPERJ (Grant \#E.200.021-2018) and Instituto Serrapilheira (Grant \# 1811-25904). GR is funded by CAPES and MS is partially supported by Arquimedes -- PUC-Rio, Brazil. This study was financed in part by the Coordena\c{c}\~ao de Aperfei\c{c}oamento de Pessoal de N\'ivel Superior - Brazil (CAPES) - Finance Code 001.

\bigskip

\noindent\textsc{Edgard A. Pimentel (Corresponding Author)}\\
Department of Mathematics\\
Pontifical Catholic University of Rio de Janeiro -- PUC-Rio\\
22451-900, G\'avea, Rio de Janeiro-RJ, Brazil\\
\noindent\texttt{pimentel@puc-rio.br}

\bigskip

\noindent\textsc{Giane C. Rampasso}\\
Department of Mathematics\\
Universidade Estadual de Campinas -- IMECC-Unicamp\\
13083-859, Cidade Universit\'aria, Campinas-SP, Brazil\\
\noindent\texttt{girampasso@ime.unicamp.br}

\bigskip

\noindent\textsc{Makson S. Santos}\\
Department of Mathematics\\
Pontifical Catholic University of Rio de Janeiro -- PUC-Rio\\
22451-900, G\'avea, Rio de Janeiro-RJ, Brazil\\
\noindent\texttt{makson.santos@mat.puc-rio.br}

\end{document}